\documentclass[11pt]{article}
\usepackage[a4paper]{anysize}\marginsize{3.5cm}{3.5cm}{1.3cm}{2cm}
\pdfpagewidth=\paperwidth \pdfpageheight=\paperheight
\usepackage{amsfonts,amssymb,amsthm,amsmath,eucal}
\usepackage{xypic}

\theoremstyle{plain}
\newtheorem{thm}{Theorem}[section]
\newtheorem{theorem}[thm]{Theorem}
\newtheorem*{theorem*}{Theorem}

\newtheorem{lemma}[thm]{Lemma}
\newtheorem{proposition}[thm]{Proposition}
\newtheorem{corollary}[thm]{Corollary}
\newtheorem{aside}[thm]{Aside}
\theoremstyle{definition}

\newtheorem{remark}[thm]{Remark}
\newtheorem{example}[thm]{Example}

\newtheorem{question}[thm]{Question}

\newtheorem{conjecture}[thm]{Conjecture}

\renewcommand\ge{\geqslant}
\renewcommand\geq{\geqslant}
\renewcommand\le{\leqslant}
\renewcommand\leq{\leqslant}
\newcommand\longto{\longrightarrow}
\newcommand\numequiv{\equiv_{\rm num}}

\newcommand\be{\begin{eqnarray*}}
\newcommand\ee{\end{eqnarray*}}
\newcommand\compact{\itemsep=0cm \parskip=0cm}
\newcommand\engqq[1]{``#1''}

\newcommand\Q{\mathbb Q}

\newcommand\newop[2]{\def#1{\mathop{\rm #2}\nolimits}}

\newop\log{log}
\newop\ord{ord}
\newop\Gal{Gal}
\newop\SL{SL}
\newop\mult{mult}
\newop\nr{nr}
\newop\Aut{Aut}
\newcommand\eqnref[1]{(\ref{#1})}
\newcommand{\olc}{\ensuremath{\Omega_X^1(\log C)}}

\begin{document}

\author{Th.~Bauer, B.~Harbourne\footnote{Brian Harbourne was partially supported by National Security Agency 
Grant/Cooperative
agreement H98230-11-1-0139, and thus
the United States Government is authorized to reproduce and distribute reprints
notwithstanding any copyright notice.}, 
A.~L.~Knutsen, A.~K\"uronya\footnote{Alex K\"uronya was partially supported by  the DFG-Forschergruppe 790
``Classification of Algebraic Surfaces and Compact Complex Manifolds'', and the OTKA Grants 77476, 77604, and 81203 by the Hungarian Academy of Sciences.},
\\S.~M\"uller-Stach, X.~Roulleau\footnote{Xavier Roulleau is a member of the project Geometria Algebrica PTDC/MAT/099275/2008 and was supported by FCT grant
SFRH/BPD/72719/2010.}, T.~Szemberg\footnote{Tomasz Szemberg was partially supported by NCN grant UMO-2011/01/B/ST1/04875.}}
\title{Negative curves on algebraic surfaces}
\date{April 3, 2012}
\maketitle
\thispagestyle{empty}

\begin{abstract}
   We study curves of negative self-intersection on algebraic surfaces.
   In contrast to what occurs in positive characteristics,
   it turns out that any smooth complex projective surface $X$
 with a surjective non-isomorphic endomorphism has bounded negativity
 (i.e., that $C^2$ is bounded below for prime divisors $C$ on $X$).
    We prove the same statement  for Shimura curves on Hilbert modular surfaces.
    As a byproduct we obtain   that there
    exist only finitely many smooth Shimura curves on a given Hilbert modular
    surface.  We also show that any set of curves of bounded genus
 on a smooth complex projective surface
 must have bounded negativity.
 \end{abstract}


\section{Introduction}

   In recent years there has been a lot of progress in
   understanding various notions and concepts of positivity
   \cite{PAG}. In the present note we go in the opposite
   direction and study negative curves on complex algebraic
   surfaces. By a negative curve we will always mean a reduced,
   irreducible curve with negative self-intersection.

The results we present here were motivated by the study of
   an old folklore conjecture, sometimes referred to as the
   Bounded Negativity Conjecture, that we may state as follows.

\begin{conjecture}[Bounded Negativity Conjecture]
   For each smooth complex projective surface $X$ there
   exists a number $b(X)\geq0$ such that
   $C^2\geq -b(X)$ for every negative curve $C\subset X$.
\end{conjecture}

The origins of this conjecture are unclear, but
it has a long oral tradition. (M.~Artin mentioned it to the second author
no later than about 1980, and we recently learned that
F.~Enriques had mentioned the conjecture to his last student,
A.~Franchetta, who in turn mentioned it to his student, C.~Ciliberto.
Ciliberto also recalls Franchetta discussing the problem
with E.~Bombieri during a trip to Naples many years ago.)
For recent references to the conjecture, see
\cite{Fon10}, \cite[Conjecture 1.2.1]{Har10}, and \cite[Question p.~24]{Har09}.
   While the occurence of smooth complex surfaces having curves
   of arbitrarily negative self-intersection still remains
   mysterious, we present here related results that arose
   from our attempts to decide the validity of the conjecture.

   It has been known for a long time that there are algebraic surfaces with
   infinitely many negative curves,
   the simplest examples being the projective plane blown up in the base
   locus of a general elliptic pencil or certain elliptic K3 surfaces. In the first
   example all negative curves have self-intersection $-1$,
   in the second example the self-intersection is $-2$, but in both cases the
   negative curves all are rational.
   In characteristic $p>0$, surfaces with negative curves
   of arbitrarily negative self-intersection have also been known for some time
   (see \cite[Exercise V.1.10]{Har77}), but the curves in these examples
   all have the same genus, and result from surjective endomorphisms
   (coming from powers of the Frobenius) of the surface
   containing them.

At this point the following questions appear to be quite natural.

   \begin{itemize}
   \item[(1)] Can one construct examples over the complex numbers of surfaces with surjective endomorphisms
   that result in negative curves
   of arbitrarily negative self-intersection?
   \item[(2)] More generally, what happens if we  replace  endomorphisms by correspondences?
   \item[(3)] Is it possible to have a surface $X$ with infinitely many
   negative curves $C$ of bounded genus such that $C^2$ is not bounded from below?
   \item[(4)] For which $d<0$ (or $g\geq0$) is it possible to produce examples of surfaces $X$
   with infinitely many negative curves $C$ such that $C^2=d$ (or such that $C$
   has genus $g$)?
   \item[(5)] If there is a lower bound for the self-intersections of negative curves on
   a given surface $X$, is there also a lower bound for the self-intersections of reduced but not
   necessarily irreducible
   curves $C$ on $X$? If so, how are the bounds related?
   \end{itemize}

   In Section~\ref{sect-bndedneg} we answer the first and third
   questions. We show that over the complex numbers  a
   surface with a non-invertible surjective endomorphism  must have bounded negativity.
   In addition, we point out that
   bounding the genus of a set of curves on a given complex
   surface of non-negative Kodaira dimension immediately leads to a lower bound on  their
   self-intersections. (The latter result was first proved in
   \cite{LS}.)

   In Section~\ref{sect-Shimura} we study the second question in the particular case
   of quaternionic Shimura surfaces, which, as is well known, carry a large infinite algebra of
   Hecke correspondences. We prove  that the negativity of Shimura curves on quaternionic Shimura
   surfaces of Hilbert modular type is bounded and that there exist only finitely many negative curves.
   This implies immediately a result that seems to have escaped the attention so far,
   namely, that there are only finitely many smooth Shimura curves on any Shimura surface of Hilbert modular type.

   In Section~\ref{sect-inf} we address the fourth question above; we verify (see Theorem \ref{ThmB}) that
   for each integer $m>0$ there is a smooth projective complex surface
   containing infinitely many smooth irreducible curves of
   self-intersection $-m$, whose
   genus can be prescribed
   when $m\ge 2$.

   Finally, in
   Section~\ref{reducible-sect} we address Question~5, by giving a sharp lower
   bound on the self-intersections of reduced curves, 
   for surfaces for which the self-intersections of negative curves are bounded below.

\paragraph*{Acknowledgement.}
   We would like to thank Jan Hendrik Bruinier, Lawrence Ein, Friedrich Hirzebruch,
   Annette Huber-Klawitter, Jun-Muk Hwang, Nick Shepherd-Barron and Kang Zuo for useful discussions.
   Special thanks go to Fritz H\"ormann, who discovered that an earlier version
   of this manuscript relied on a published result that turns out to be false. This collaboration
   grew out of interactions during the Oberwolfach mini-workshop
   \engqq{Linear Series on Algebraic Varieties}. We thank the MFO for excellent working
   conditions. We also thank SFB/TRR 45 for financial support for visits among some of us and for
   the 2009 summer school in Krak\'ow, where some initial consideration of these problems occurred.


\section{Bounded Negativity}\label{sect-bndedneg}

   In positive characteristic there exist surfaces carrying a
   sequence of irreducible curves with self-intersection tending
   to negative  infinity (see \cite[Exercise V.1.10]{Har77}). These
   curves are constructed by taking iterative images of a negative
   curve under a surjective endomorphism of the surface.

   In more detail,
   the construction goes as follows. Let $C$ be a curve of
   genus $g\geq 2$ defined over an algebraically closed field $k$
   of characteristic $p$, let $X=C\times C$ be the product
   surface with  $\Delta\subset X$ the diagonal.
   Furthermore let $F:C\to C$ be the Frobenius homomorphism,
   defined by taking coordinates of a point on $C$ to their $p$-th powers.
   Then $G=id\times F$ is a surjective endomorphism of $X$.
   The self-intersections in the sequence of irreducible curves
   $\Delta, G(\Delta), G^2(\Delta),\dots$ tend to negative infinity.

   We now show that in characteristic zero
   it is not possible to construct a sequence of curves with
   unbounded negativity using endomorphisms as above.
   In fact we prove an even stronger statement: the existence
   of a non-trivial surjective endomorphism implies a bound
   on the negativity of self-intersections of curves on the surface.

\begin{proposition}\label{firstprop}
   Let $X$ a smooth projective complex surface admitting a surjective
   endomorphism that is not an isomorphism. Then $X$ has bounded
   negativity, i.e., there is a bound $b(X)$ such that
   $$
      C^2 \ge -b(X)
   $$
   for every reduced irreducible curve $C\subset X$.
\end{proposition}

\begin{proof}
   It is a result of Fujimoto and Nakayama
   (\cite{Fuj02} and \cite{Nak02})
   that a surface $X$
   satisfying our  hypothesis is of one of the following
   types:
   \begin{itemize}\compact
   \item[(1)]
      $X$ is a toric surface;
   \item[(2)]
      $X$ is a $\mathbb P^1$-bundle;
   \item[(3)]
      $X$ is an abelian surface or a hyperelliptic surface;
   \item[(4)]
      $X$ is an elliptic surface with Kodaira dimension
      $\kappa(X)=1$ and topological Euler number $e(X)=0$.
   \end{itemize}
   In cases (1) and (2) the assertion is clear as $X$
   then carries only finitely many negative curves.
   In case (3) bounded negativity follows from
   the adjunction formula (cf.~\cite[Prop.~3.3.2]{LS}).
   Finally, bounded negativity for elliptic surfaces
   with Euler number zero
   will be established in
   Proposition~\ref{prop-elliptic}.
\end{proof}

\begin{proposition}\label{prop-elliptic}
   Let $X$ be a smooth projective complex  elliptic surface with $e(X)=0$. Then
   there are no negative curves on $X$.
\end{proposition}

\begin{proof}
   Let $\pi:X\to B$ be an elliptic fibration,
   where
   $B$ is a smooth curve, and let $F$ be the class of a fiber of
   $\pi$. By the properties of $e(X)$ of a fibered surface
   (cf.~\cite[III, Proposition~11.4 and Remark~11.5]{BPV}),
   the only singular fibers of $X$ are possible multiple fibers,
   and the reduced fibers are always smooth elliptic curves.
   In particular, $X$ must be minimal and its fibers do not contain negative curves.

   Aiming at a contradiction, assume that $C\subset X$ is a negative curve. Then, by the above, the intersection number
   $n:=C\cdot F$ is positive. This means that $\pi$
   restricts to
   a map
   $C\to B$ of degree $n$.
   Taking an embedded resolution $f:\widetilde X\to X$
   of $C$, we get
   a smooth curve
   $\widetilde C=f^*C-\Gamma$, where the divisor $\Gamma$ is supported on the
   exceptional locus of $f$.
   The Hurwitz formula, applied to the induced covering
   $\widetilde C\to B$,
    yields
   \begin{equation}\label{eqn-hurwitz}
      2g(\widetilde C)-2=n\cdot(2g(B)-2)+\deg R \ ,
   \end{equation}
   where $R$ is the ramification divisor.

   Let $m_1F_1,\ldots,m_kF_k$ denote the multiple fibers of $\pi$.
   The assumption $e(X)=0$ implies via Noether's formula that
   $K_X\numequiv (2g(B)-2)F+ \sum(m_i-1)F_i$. Hence
   \begin{eqnarray*}
     K_X \cdot C &  =   & n(2g(B)-2)+ \sum(m_i-1)F_i \cdot C \\
                 &  =   & n(2g(B)-2)+ \sum(m_i-1)f^*F_i \cdot f^*C \\
                 &  =  &  n(2g(B)-2)+ \sum(m_i-1)f^*F_i \cdot \widetilde C \\
                 & \leq & n(2g(B)-2)+ \deg R.
   \end{eqnarray*}
   On the other hand,
   \begin{eqnarray*}
      2g(\widetilde C)-2 = \widetilde C^2+K_{\widetilde X}\cdot\widetilde C
         &=& (f^*C-\Gamma)^2+(f^*K_X+K_{\widetilde X/X})(f^*C-\Gamma) \\
         &=& C^2+\Gamma^2+K_X\cdot C-K_{\widetilde X/X}\cdot\Gamma.
   \end{eqnarray*}
   Consequently, using \eqref{eqn-hurwitz}, we obtain
   $$
      C^2 \geq K_{\widetilde X/X}\cdot\Gamma-\Gamma^2.
   $$
   The subsequent lemma yields the contradiction $C^2 \geq 0$.
\end{proof}

\begin{lemma}
   Let $f:Z\to X$ be a birational morphism of smooth
   projective surfaces, and let $C\subset X$ be any curve, with
   proper transform $\widetilde C=f^*C-\Gamma_{Z/X}$ on $Z$. Then
   $$
      K_{Z/X}\cdot\Gamma_{Z/X}-\Gamma_{Z/X}^2\ge 0.
   $$
\end{lemma}

\begin{proof}
   As $f$ is a finite composition of blow-ups, this
   can be seen by an elementary inductive argument.
   For the convenience of the reader we briefly indicate it.
   Suppose that $f$ consists of $k$ successive blow-ups.
   For $k=1$ the assertion is clear, since then
   $K_{Z/X}$ is the exceptional divisor $E$, and $\Gamma$
   is the divisor $mE$, where $m$ is the multiplicity of $C$ at
   the blown-up point. For $k>1$ we may decompose $f$ into two
   maps
   $$
      Z\stackrel g\longto Y\stackrel h\longto X.
   $$
   One has proper transforms
   $$
      C'=h^*C-\Gamma_{Y/X}
      \quad\mbox{and}\quad
      \widetilde C=g^*C'-\Gamma_{Z/Y}=f^*C-\Gamma_{Z/X}.
   $$
   The equalities
   \be
      K_{Z/X}&=&K_{Z/Y}+g^*K_{Y/X} \\
      \Gamma_{Z/X}&=&\Gamma_{Z/Y}+g^*\Gamma_{Y/X}
   \ee
   then imply
   \be
      K_{Z/X}\cdot\Gamma_{Z/X}-\Gamma_{Z/X}^2
         &=&(K_{Z/Y}+g^*K_{Y/X})(\Gamma_{Z/Y}+g^*\Gamma_{Y/X})
         -(\Gamma_{Z/Y}+g^*\Gamma_{Y/X})^2 \\
      &=& (K_{Z/Y}\cdot \Gamma_{Z/Y}
         - \Gamma_{Z/Y}^2)
         + (K_{Y/X}\cdot \Gamma_{Y/X}
         - \Gamma_{Y/X}^2),
   \ee
   and the assertion follows by induction.
\end{proof}

We now consider Question 3 of the introduction.
The first general result known to us answering this question
   is due to Bogomolov. It says that on a surface $X$ of general type
   with $c_1^2(X)>c_2(X)$ curves of a fixed geometric genus lie in a bounded
   family. This implies of course that their numeric invariants, in particular their self-intersections,
   are bounded. An effective version of Bogomolov's result
   was obtained by Lu and Miyaoka \cite[Theorem 1 (1)]{LuMiy95}.
   Their proof relies on Corollary \ref{cor:LMY}. We state here a more
   general result due to Miyaoka \cite[Theorem 1.3 i), ii)]{Miyaoka}, as we need it anyway in the next section.
\begin{theorem}\label{pro:By-[Miyaoka,-The}
   Let $X$ be a surface of non-negative Kodaira dimension and let $C$ be an irreducible
   curve of geometric genus $g$ on $X$. Then
\begin{equation}\label{eq:M-orbifold}
\frac{\alpha^{2}}{2}(C^{2}+3CK_{X}-6g+6)-2\alpha(CK_{X}-3g+3)+3c_{2}-K_{X}^{2}\,\geq\, 0
\end{equation}
   for all $\alpha\in[0,1]$.

   Moreover, if $C\not\simeq\mathbb{P}^{1}$,
   and $K_{X}C>3g-3$ then
\begin{equation}\label{eq:M-orbifold non rational}
2(K_{X}C-3g+3)^{2}-(3c_{2}-K_{X}^{2})(C^{2}+3CK_{X}-6g+6)\leq0.
\end{equation}
\end{theorem}
   Putting $\alpha=1$ in \eqnref{eq:M-orbifold}, we recover the classical
   logarithmic Miyaoka-Yau inequality
   (see also \cite[Appendix]{LS} for a complete direct proof).
\begin{corollary}[Logarithmic Miyaoka-Yau inequality]\label{cor:LMY}
   Let $X$ be a smooth projective surface of non-negative
   Kodaira dimension and let $C$ be a smooth curve
   on $X$. Then
   $$
      c_1^2(\olc)\leq 3c_2(\olc),
   $$
   equivalently $(K_X+C)^2\leq 3\left(c_2(X)-2+2g(C)\right)$.
\end{corollary}

   We recall
   here a statement that is numerically slightly weaker than the result of
   Lu and Miyaoka \cite[Theorem 1 (1)]{LuMiy95} but which has a simpler proof.
   This result appeared first in
   \cite[Proposition~3.5.3]{LS}, and we refer to that article for a more detailed exposition.
\begin{theorem}[Proposition~3.5.3 of \cite{LS}]\label{thm:WBNC}
   Let $X$ be a smooth projective surface with $\kappa(X)\geq 0$.
   Then for every reduced, irreducible curve $C\subset X$ of geometric genus $g(C)$ we have
   \begin{equation}\label{eq:weakbound}
      C^2\geq c_1^2(X)-3c_2(X)+2-2g(C).
   \end{equation}
\end{theorem}
   The proof is a combination of Corollary \ref{cor:LMY} and the following simple
   lemma on the behavior of \eqnref{eq:weakbound} under blow ups.
\begin{lemma}\label{lem:blowinv}
   Let $X$ be a smooth projective surface, $C\subset X$ a reduced, irreducible
   curve of geometric genus $g(C)$, $P\in C$ a point with $m:=\mult_PC\geq 2$.
   Let $\sigma:\widetilde{X}\to X$ be the blow up of $X$ at $P$ with the
   exceptional divisor $E$. Let $\widetilde{C}=\sigma^*(C)-mE$ be the proper
   transform of $C$. Then the inequality
   $$\widetilde{C}^2\geq c_1^2(\widetilde{X})-3c_2(\widetilde{X})+2-2g(\widetilde{C})$$
   implies
   $$
      C^2\geq c_1^2(X)-3c_2(X)+2-2g(C).
   $$
\end{lemma}
\begin{proof}
   This follows by direct computation using the facts that
   $C^2=\widetilde{C}^2+m^2$, $c_1^2(X)=c_1^2(\widetilde{X})+1$, $c_2(X)=c_2(\widetilde{X})-1$ and $g(C)=g(\widetilde{C})$.
\end{proof}

\begin{proof}[Proof of Theorem \ref{thm:WBNC}]
    Taking an embedded resolution $f:\widetilde X\to X$
   of $C$ and applying Lemma \ref{lem:blowinv} to every step, we reduce to proving the assertion for $C$ smooth.

The latter case easily follows from Corollary \ref{cor:LMY}.
   Indeed, our assumption $\kappa(X)\geq 0$ implies that $K_{\widetilde{X}}+\widetilde{C}$
   is $\Q$--effective. Hence we have
   \be
      c_1^2(X)+2C\cdot(K_X+C)-C^2&=&c_1^2(\olc)\\
      &\leq& 3c_2(\olc)=3c_2(X)-6+6g(C).
   \ee
   Rearranging terms and using the adjunction formula, we arrive at \eqnref{eq:weakbound}.
\end{proof}
   A closer analysis of Corollary \ref{cor:LMY} allows one to ease the
   assumption of $X$ being of non-negative Kodaira dimension by the
   assumption of $X$ being of non-negative logarithmic Kodaira dimension,
   see \cite[Corollary 1.2]{Miy84}.
\begin{aside}[Strong Logarithmic Miyaoka-Yau inequality]\label{asi:SLMY}
   Let $X$ be a smooth projective surface and $C$ a smooth curve
   on $X$ such that the adjoint line bundle $K_X+C$ is $\Q$--effective,
   i.e., there is an integer $m>0$ such that $h^0(m(K_X+C))>0$. Then
   $$
      c_1^2(\olc)\leq 3c_2(\olc),
   $$
   equivalently $(K_X+C)^2\leq 3\left(c_2(X)-2+2g(C)\right)$.
\end{aside}
   So one gets the same bound \eqnref{eq:weakbound} as in Theorem \ref{thm:WBNC},
   for all curves $C$ such that $K_X+C$ is $\Q$--effective.

\section{Negativity of Shimura curves on Hilbert modular surfaces}\label{sect-Shimura}
\subsection{Smoothness of Shimura curves and Hecke translates}\label{Section : Criteria}
   We begin by giving  a criterion for a Shimura curve to be smooth
  on a quaternionic Hilbert modular surface
   and indicating why its Hecke translates might fail to remain smooth.
   In the next subsection we will show that the worst scenario actually happens.

   We recall first how (quaternionic) Shimura surfaces
   are defined. For a complete reference on their construction see \cite{Deligne},
   and for particularly interesting examples see \cite{Granath} and \cite{Shavel}.
   Let $A$ be a ramified quaternion algebra over a totally real number field
$k$. Let $\mathcal{O}_{A}$ be a maximal order of $A$ and let
\[
\Gamma(1)=\{\gamma\in\mathcal{O}_{A}:\;\; \nr(\gamma)=1\},
\]
where $\nr$ denotes  the reduced norm. Suppose that $A$ splits over exactly
two places in $\mathbb{R}$, i.e., there exist two embeddings $\sigma_{i}:k\to\mathbb{R}$
such that the tensor products $A\otimes_{k^{\sigma_{i}}}\mathbb{R}$
over these places are isomorphic to $M_{2}(\mathbb{R})$, while for all the
other embeddings the tensor product is isomorphic to the Hamiltonian
quaternions.

We fix such isomorphisms, giving rise to  a representation
\[
\begin{array}{cccc}
\rho: & A & \to & M_{2}(\mathbb{R})\times M_{2}(\mathbb{R})\\
 & \gamma & \to & (\gamma_{1},\gamma_{2})\end{array}\ .
\]
The morphism $\rho$ maps $A^{\times}$ into $GL_{2}(\mathbb{R})^{2}$.
Let $A^{+}$ be the sub-group of elements $\gamma$ of $A$ such that
$\det(\gamma_{i})>0$ for $i=1,2$. The group $A^{+}$ acts on $\mathbb{H}\times\mathbb{H}$
by
\[
\gamma\cdot(z_{1},z_{2})=(\gamma_{1}\cdot z_{1},\gamma_{2}\cdot z_{2}) \,,
\]
where, for $\gamma_{i}=\left(\begin{array}{cc}
a & b\\
c & d\end{array}\right)$, we have
\[
\gamma_{i}\cdot z=\frac{az+b}{cz+d}\,.
\]

Let us denote by $\Gamma$ a sub-group of $A^{+}$  commensurable to $\Gamma(1)$, i.e., $\Gamma(1)\cap\Gamma$ has finite index in
both $\Gamma(1)$ and $\Gamma$. With these hypotheses, the quotient
$X=\mathbb{H}\times\mathbb{H}/\Gamma$ is a compact algebraic surface
(see \cite{Deligne}), called a  Shimura surface (or a compact Hilbert modular surface).

Let us suppose in addition that $\Gamma$ is torsion free or, equivalently, that  $X$ is smooth. The surface $X$ is then minimal of general
type with  $c_{1}^{2}=2c_{2},\, q=0$. We denote by
 $\pi:\mathbb{H}\times\mathbb{H}\to X$ the quotient map.

A Shimura curve is, in particular, a totally geodesic curve in $X$.
Let $C'_{1}$ be such a Shimura curve on $X$ and let
\[
\mathbb{H}_{1}\subset\pi^{-1}C_{1}'\subset\mathbb{H}\times\mathbb{H}
\]
 be a subspace isomorphic to $\mathbb{H}$ so that
\[
\Lambda_{1}=\{\gamma\in\Gamma:\;\; \gamma\mathbb{H}_{1}=\mathbb{H}_{1}\}
\]
is a lattice in $\Aut(\mathbb{H}_{1})$. Then $C_{1}=\mathbb{H}_{1}/\Lambda_{1}$
is a smooth compact curve whose image under  the generically one-to-one
map $C_{1}\rightarrow X$ we call $C_{1}'$.
\begin{proposition}
\label{proposition Smoothness criteria}The Shimura curve $C'_{1}$
is smooth if and only if $\mathbb{H}_{1}\cap\gamma\mathbb{H}_{1}=\emptyset$ for all $\gamma\in\Gamma\setminus\Lambda_{1}$.
\end{proposition}

\begin{proof}
The map $C_{1}\rightarrow X$ is an immersion because the map $\mathbb{H}_{1}\to X$
is so. Thus singularities on $C'_{1}$ can occur if and only if there
are two distinct points $\Lambda_{1}t,\,\Lambda_{1}u$ on $C_{1}$
(with $u,t\in\mathbb{H}_{1}$) mapped onto the same point by the generically
one-to-one map $C_{1}\to C'_{1}$. For such points we have
$\Gamma t=\Gamma u$, i.e., there exist $\gamma\in\Gamma$ such that $t=\gamma u$. As $\Lambda_{1}t\not=\Lambda_{1}u$,
we have $\gamma\in\Gamma-\Lambda_{1}$
and the intersection of the upper-halfplanes $\mathbb{H}_{1}$ and $\gamma\mathbb{H}_{1}$
is not empty. Conversely, if the intersection of the upper-halfplanes
$\mathbb{H}_{1}$ and $\gamma\mathbb{H}_{1}$ is not empty, there
are two distinct points on $C_{1}$ that have the same image on $C'_{1}$,
and thus there is a singularity on $C_{1}'$.
\end{proof}

For $h\in A^{+}$, set $\mathbb{H}_{h}:=h(\mathbb{H}_{1})$ and
let
\[
\Lambda_{h}=\{\lambda\in\Gamma:\;\;\lambda\mathbb{H}_{h}=\mathbb{H}_{h}\}\ .
\]
The group $\Lambda_{h}$ is equal to the lattice $h\Lambda_{1}h^{-1}\cap\Gamma$.
Let $C_{h}=\mathbb{H}_{h}/\Lambda_{h}$ and let $C'_{h}$ be the image of $C_{h}$ in $X$ under the natural map. Again, $C'_{h}$
is a Shimura curve.

\begin{proposition}
\label{prop. Shimura-curves C_1 et C_h}Suppose that the  curve
$C'_{1}$ is smooth. Then the Shimura curve $C'_{h}$ is smooth if and
only if $\mathbb{H}_{1}\cap\gamma\mathbb{H}_{1}=\emptyset$ for all $\gamma\in h^{-1}\Gamma h\setminus\Gamma$.
\end{proposition}

\begin{proof}
We apply Proposition \ref{proposition Smoothness criteria} to $C'_{h}$.
The curve $C'_{h}$ is smooth if and only if $\mathbb{H}_{h}\cap\gamma\mathbb{H}_{h}=\emptyset$ for all $\gamma\in\Gamma\setminus\Lambda_{h}$.
 Suppose that the curve $C'_{h}$ is singular. There exist then $z_{1},z_{2}\in\mathbb{H}_{1}$
(whence $hz_{1},hz_{2}\in\mathbb{H}_{h}$) and $\gamma\in\Gamma\setminus\Lambda_{h}$
such that $hz_{1}=\gamma(hz_{2})$. Then $z_{1}=h^{-1}\gamma hz_{2}$.
As $C'_{1}$ is smooth, we have two possibilities: either $h^{-1}\gamma h\in\Lambda_{1}$
and $h^{-1}\gamma h\not\in\Gamma$ or for
$\gamma'=h^{-1}\gamma h\in h^{-1}\Gamma h\setminus\Gamma$ we have
$\mathbb{H}_{1}\cap\gamma\mathbb{H}_{1}\not=\emptyset$.
The first possibility is impossible
because $\Lambda_{h}=h\Lambda_{1}h^{-1}\cap\Gamma$ and we assumed that
$\gamma\in\Gamma\setminus\Lambda_{h}$. Therefore the second possibility holds.
For the converse statement, we remark that all the above arguments are in
fact equivalences.
\end{proof}

As we will remark below, each element $h$ of $A^{+}$ defines a Hecke correspondence
$T_{h}$, and the curve $C'_{h}$ is an irreducible component of the
image of $C'_{1}$ by $T_{h}$. We have $T_{h}=T_{h'}$ if and only
if $\Gamma h=\Gamma h'$.
When varying $\Gamma h$ in $\Gamma\setminus A^{+}$, we see by the
above Proposition \ref{prop. Shimura-curves C_1 et C_h} that in order to keep $C'_{h}$
smooth, the half-plane $\mathbb{H}_{1}$ must avoid more and more
half-planes $\gamma\mathbb{H}_{1}$. Our next result (Proposition \ref{pro:We-have-:For}) shows that this is only possible
in  finitely many cases.

Let us now explain how Hecke correspondences come into the game. For
$h\in A^{+}$, let
\[
\Gamma_{h}=\Gamma\cap h^{-1}\Gamma h,
\]
which  is a subgroup of finite index $m$ in $\Gamma$. Let $t_{1}=1,t_{2},\dots,t_{m}$
be a full set of coset  representatives of $G$  with respect to $\Gamma_{h}$. Denote by  $X_{h}$
be the Shimura surface $X_{h}=\mathbb{H}\times\mathbb{H}/\Gamma_{h}$.

There are two \' etale maps  of degree $m$,
 \[
\xymatrix{
X_{h} \ar[r]^{\pi_{1}} \ar[d]_{\pi_{2}} &  X \\
X & }
 \]
where $\pi_{1}(\Gamma_{h}.z)=\Gamma.z$ and $\pi_{2}(\Gamma_{h}.z)=\Gamma h.z$.
We need to check that $\pi_{2}$ is well defined. Let $\tau:=h^{-1}\gamma h\in\Gamma_{h}$
with $\gamma\in\Gamma$ and $z':=\tau z$. Then \[
\Gamma h.z'=\Gamma h\tau.z=\Gamma hh^{-1}\gamma h.z=\Gamma\gamma h.z=\Gamma h.z\,,\]
and therefore the map $\pi_{2}$ does not depend on the choice of a representative
in $\Gamma_{h}.z$.  The Hecke operator $T_{h}$ is defined by $T_{h}=\pi_{2*}\pi_{1}^{*}$.
We have $\pi_{1}^{-1}\Gamma z=\Gamma_{h}t_{1}.z+\dots+\Gamma_{h}t_{m}.z$
and \[
T_{h}(\Gamma.z)=\Gamma ht_{1}.z+\dots+\Gamma ht_{m}.z\,.\]
It follows that $T_{h}C'_{1}=C'_{h}+Y_{2}+\dots+Y_{t}$ for some
irreducible curves $Y_{2},\dots,Y_{t}$.

\begin{remark}
Let  $X$ be a smooth Picard surface, i.e.,  $X= \mathbb{B}_2 /\Gamma $ is a quotient of the unit complex $2$-dimensional ball
$\mathbb{B}_2$ by a co-compact torsion free group $\Gamma \subset PU(2,1)$. It is possible to obtain the same  results (smoothness criteria, smoothness
of the Hecke translates) for a Shimura curve $C=\mathbb{B}_1/ \Lambda $ on $X$.  Again, the main idea  is that in order for an irreducible
component of the translate $T_h C$ of a Shimura curve to be smooth, the ball $h \mathbb{B}_1$  must avoid more and more balls when $h$ varies.
\end{remark}


\subsection{Finiteness of smooth Shimura curves.}

Let $X$ be a compact Hilbert modular surface. As we will see, the self-intersection of a smooth Shimura curve $C$ on $X$ is very negative, in particular
 $C^{2}=-(2g(C)-2)<0$. On the other hand,
the set of Shimura curves on $X$ is preserved by Hecke correspondences.
It is therefore very natural to hope to obtain a counterexample to
the bounded negativity conjecture by taking the images of a Shimura
curve by Hecke correspondences. We will see however that there is only a finite number of Shimura curves (smooth or not) with $C^2 <0.$

Let $C$ be a curve on $X$ of geometric genus $g$. The difference
\begin{equation}\label{eq:delta}
\delta=\frac{1}{2}(K_{X}\cdot C+C^{2}-2g+2)
\end{equation}
   where $K_{X}$ is the canonical divisor of $X$, is a positive integer.
   If the curve is nodal, then this equals  the number of nodes on $C$.
   We recall the following important Theorem from \cite{Barthel}.
\begin{theorem}[Hirzebruch-H\"ofer Proportionality Theorem]\label{lem:By-[Autissier,-Gasbarri,}
   For a Shimura curve $C$ on a quaternionic Hilbert modular surface $X$ we have
$$
K_{X}C=4(g-1)\;\;\mbox{ and }\;\; K_{X}C+2C^{2}=4\text{\ensuremath{\delta}}\,.$$
\end{theorem}

Although we will not use this fact, it is interesting to notice that the curve $C$ in 
Theorem \ref{lem:By-[Autissier,-Gasbarri,} is nodal.

   The main
result of this section is
\begin{proposition}
\label{pro:We-have-:For}For a Shimura curve $C$ on $X$
we have the  following inequalities
\[
g \leq 1+c_{2}+\sqrt{c_{2}^{2}+c_{2}\delta}\;\;\mbox{ and }\;\; C^{2}\geq-6c_{2}\ .\]
In particular, if $C$ is smooth, then $g\leq1+2c_{2}$.

Moreover, there is
only a finite number of Shimura curves with $C^{2}<0$, since for $\delta\geq3c_{2}$,
the curve $C$ satisfies $C^{2}\geq0$.
\end{proposition}

\begin{proof}
   The idea is to show that for $\delta\geq3c_{2}$,
the Shimura curve $C$ satisfies $C^{2}\geq0$. Computing $g$ from \eqnref{eq:delta}
and inserting it into \eqnref{eq:M-orbifold} we obtain
\[
P(\alpha)=\alpha^{2}(3\delta-C^{2})+\alpha(CK_{X}+3C^{2}-6\delta)+3c_{2}-K_{X}^{2}\,\geq\, 0\]
for $0\leq\alpha\leq1$. Using the second equality in Theorem \ref{lem:By-[Autissier,-Gasbarri,}
and since $K_{X}^{2}=2c_{2}$ for compact Hilbert modular surfaces, we  get
\[
P(\alpha)=\alpha^{2}(3\delta-C^{2})+\alpha(C^{2}-2\delta)+c_{2} \,\geq\, 0.\]
If $C^{2}\geq2\delta$, then obviously $C^{2}\geq0$. If $C^{2}<2\delta$,
the minimum of $P(\alpha)$ is attained for \[
\alpha_{0}=\frac{2\delta-C^{2}}{2(3\delta-C^{2})}.\]
Note that $0<\alpha_{0}<1$. Evaluating the condition $P(\alpha_0)\geq 0$,
we obtain
\begin{equation}\label{eq:roots of P}
2c_{2}+2\sqrt{c_{2}^{2}+\delta c_{2}}\geq2\delta-C^{2}\geq2c_{2}-2\sqrt{c_{2}^{2}+\delta c_{2}}.
\end{equation}
For $C^{2}<2\delta$, we get the lower bound\[
C^{2}\geq2\delta-2c_{2}-2\sqrt{c_{2}^{2}+\delta c_{2}}\,.\]
Hence, if $\delta\geq3c_{2}$, we indeed have $C^{2}\geq0$.

Suppose now that $C^{2}<0$. Then $\delta<3c_{2}$, and therefore
$-2\sqrt{c_{2}^{2}+\delta c_{2}}>-4c_{2}.$
 We get from \eqnref{eq:roots of P} that $$C^{2}\geq2\delta-2c_{2}-2\sqrt{c_{2}^{2}+\delta c_{2}}\geq2(\delta-3c_{2})$$
and consequently $C^{2}\geq-6c_{2}$.

Miyaoka's formula \eqnref{eq:M-orbifold non rational} with $C^{2}+K_{X}C=2g-2+2\delta$ implies\[
(K_{X}C-3g+3)^{2}-c_{2}(K_{X}C+\delta-2g+2)\leq0\,.\]
As $K_{X}C=4g-4$, we get \[
(g-1)^{2}-2c_{2}(g-1)-c_{2}\delta\leq0\]
and therefore \[
g-1\leq c_{2}+\sqrt{c_{2}^{2}+c_{2}\delta}\,.\]
Now for $C^{2}<0$, we know that $\delta<3c_{2}$ and thus we have $g\leq3c_{2}+1$.
Since $K_X C=4g-4$, the intersection number $K_X C$ is bounded from above. 
An infinite number of Shimura curves with bounded geometric genus $g$ and bounded intersection with $K_X$  
must be in a finite number of families of curves, thus these Shimura curves must
deform and satisfy $C^{2}\geq0$, therefore the number of Shimura
curves with $C^{2}<0$ must be finite.
\end{proof}

\begin{corollary}\label{cor:finitely many smooth Shimuras}
   There are only finitely many smooth Shimura curves on a compact Hilbert modular surface.
\end{corollary}
\proof
   This follows immediately from Proposition \ref{pro:We-have-:For}, as
   smooth Shimura curves have a negative self-intersection by the
   second equality in Theorem \ref{lem:By-[Autissier,-Gasbarri,}.
\endproof
\begin{remark}
   It is easy to see that the compactness of $X$ was not used in the course
   of the proof of Proposition \ref{pro:We-have-:For}. The same statement holds
   therefore for open Hilbert modular surfaces. An important ingredient of the proof
   was however the Proportionality property \ref{lem:By-[Autissier,-Gasbarri,}.
   In fact the Proportionality holds also for modular curves (i.e., those passing through
   the cusps of a Hilbert modular surface) \cite[Theorems 0.1 and 0.2 combined]{MVZ09},
   hence the statement of
   Proposition \ref{pro:We-have-:For} remains valid for such curves.

   Since the numerics are different for ball quotients, we do not know
   if there is a bound similar to that of Proposition \ref{pro:We-have-:For}
   for ball quotients.
\end{remark}


\section{Surfaces with infinitely many negative curves of fixed self-intersection}\label{sect-inf}

   The well-known example of $\mathbb P^2$ blown-up at nine points
   shows that there are surfaces containing infinitely many
   $(-1)$-curves. Along similar lines, we point out here that one can exhibit  surfaces with
   infinitely many negative curves of any given (fixed) negative
   self-intersection.

\begin{theorem}
   For every integer $m>0$ there are smooth projective complex surfaces
   containing infinitely many smooth irreducible curves of
   self-intersection $-m$.
\end{theorem}

\begin{proof}
   Let $E$ be an elliptic curve without complex multiplication,
   and let $A$ be the abelian surface $E\times E$. We denote by
   $F_1$ and $F_2$ the fibers of the projections and by $\Delta$
   the diagonal in $A$. It is shown in
   \cite[Proposition~2.3]{BauSch08} that every elliptic curve on
   $A$ that is not a translate of $F_1,F_2$ or $\Delta$
   has numerical
   equivalence class of the form
   $$
      E_{c,d}:=c(c+d)F_1+d(c+d)F_2-cd\Delta,
   $$
   where $c$ and $d$ are suitable coprime integers, and
   conversely, that every such numerical class corresponds to an
   elliptic curve $E_{c,d}$ on $A$. In our construction we will
   make use of a sequence $(E_n)$ of such curves, for instance
   taking $E_n=E_{n,1}$ for $n\ge 2$. No two of the curves $E_n$
   are then translates of each other.

   Fix a positive integer $t$ such that $t^2\ge m$. For each of
   the elliptic curves $E_n$, the number of $t$-division points on
   $E_n$ is $t^2$, and these points are among the $t$-division
   points of $A$. (Actually, the latter is only true if $E_n$ is a
   subgroup of $A$, but this can be achieved by using a translate
   of $E_n$ passing through the origin.) Since the number of
   $t$-division points on $A$ is finite -- there are exactly $t^4$
   of them -- there must exist a subsequence of $(E_n)$ having the
   property that all curves $E_n$ in the subsequence have the same
   set of $t$-division points, say $\{e_1,\dots,e_{t^2}\}$.

   Consider now the blow-up $f:X\to A$ at
   the set $\{e_1,\dots,e_m\}$. The proper transform $C_n$ of
   $E_n$ is then a smooth irreducible curve on $X$ with
   $$
      C_n^2=E_n^2-m=-m,
   $$
   as claimed.
\end{proof}

\begin{remark}
   Note that the proof yields a one-dimensional family of
   surfaces, and that the constructed surfaces are of Picard
   number $m+3$.
\end{remark}

   For each $m\geq1$, the proof above gives a surface $X$ with
   infinitely many curves of genus 1 of self-intersection $-m$. This raises the question of whether
   for each $m\geq1$ and each $g\geq0$ there is a surface $X$ with
   infinitely many curves of genus $g$ of self-intersection $-m$.
   We now show that the answer is yes at least for $m>1$.

\begin{theorem}\label{ThmB}
   For each $m>1$ and each $g\geq0$ there exists a smooth
   projective complex surface
   containing infinitely many smooth irreducible curves of
   self-intersection $-m$ and genus $g$.
\end{theorem}

\begin{proof}
   Let $f:X\to B$ be a smooth complex projective minimal elliptic surface with section,
   fibered over a smooth base curve $B$ of genus $g(B)$. Then $X$ can have no multiple fibers, so that by Kodaira's well-known result (cf. \cite[V,Corollary~12.3]{BPV}), $K_X$ is a sum of a specific choice of $2g(B) -2 + \chi({\mathcal O}_X)$ fibers of the elliptic fibration. Let $C$ be any section of the elliptic fibration $f$. By adjunction, $C^2 = - \chi({\mathcal O}_X)$.

   Take $X$ to be rational and $f$ to have infinitely many
   sections; for example, blow up the base points of a general pencil of plane cubics.
   Then $\chi({\mathcal O}_X) = 1$, so that $C^2 = -1$ for
   any section $C$.

   Pick any $g \geq 0$ and any $m\geq 2$. Then, as is well-known \cite{Hur91}, there
   is a smooth projective curve $C$ of genus $g$ and a finite morphism $h: C \to B$ of degree $m$ that is
   not ramified over points of $B$ over which the fibers of $f$ are singular.
   Let $Y = X \times_B C$ be the fiber product. Then the projection $p:Y \to C$
   makes $Y$ into a minimal elliptic surface, and each section of $f$ induces a section of
   $p$. By the property of the ramification of $h$, the surface $Y$ is smooth and
   each singular fiber of $f$ pulls back to $m$ isomorphic singular fibers of $p$.
   Since $e(Y)$ is the sum of the Euler characteristics of the
   singular fibers of $p$ (cf. e.g. \cite[III, Proposition~11.4]{BPV}), we obtain from Noether's formula that
   $\chi({\mathcal O}_Y) = e(Y)/12 = me(X)/12 = m\chi({\mathcal O}_X) = m$. Therefore,  for any section
   $D$ of $p$, we have $D^2 = -m$; i.e., $Y$ has infinitely many smooth irreducible curves of genus $g$ and
   self-intersection $-m$.
\end{proof}

\begin{question}
   Is there for each $g>1$ a surface with infinitely many $(-1)$-curves of
   genus $g$?
\end{question}


\section{Negativity of reducible curves}\label{reducible-sect}

   When asking for bounded negativity of curves, it is necessary
   to restrict attention to reduced curves. Irreducibility, however, is not
   an essential hypothesis,
   since by \cite[Proposition 3.8.2]{LS},
   bounded negativity holds for the set of reduced, irreducible curves on a surface $X$
   if and only if it holds for the set of reduced curves on $X$. Here we improve
   this result by obtaining a sharp bound on the negativity for reducible curves, given a bound
   on the negativity for reduced, irreducible curves.

\begin{proposition}
   Let $X$ be a smooth projective surface (over an arbitrary algebraically closed ground field)
   for which there is a
   constant $b(X)$ such that $C^2\ge -b(X)$ for every reduced,
   irreducible curve $C\subset X$. Then
   $$
      C^2\ge -(\rho(X)-1)\cdot b(X)
   $$
   for every reduced curve $C\subset X$, where $\rho(X)$ is the
   Picard number of $X$.
\end{proposition}

\begin{proof}
   Consider the Zariski decomposition $C=P+N$ of the reduced
   divisor $C$. Then $C^2 = P^2+N^2 \ge N^2$, as $P$ is nef and
   $P$ and $N$ are
   orthogonal. So the issue is to bound $N^2$. The
   negative part $N$ is of the form $N=a_1 C_1 + \cdots + a_r C_r$,
   where the curves $C_i$ are among the components of $C$ and
   the coefficients $a_i$ are positive rational numbers.
   Note that
   $a_i\le 1$ for all $i$, because $C$ is reduced. Since the intersection
   matrix of $N$ is negative definite, we have $r\le\rho(X)-1$.
   Thus
   $$
      C^2\ge N^2 \ge a_1^2 C_1^2 + \cdots  + a_r^2 C_r^2 \ge -r\cdot b(X)
      \ge -(\rho(X)-1)\cdot b(X),
   $$
   as claimed.
\end{proof}

\begin{example}\rm
   Here is an example of a surface of higher Picard number, for which
   equality holds in the inequality
   $C^2\ge -(\rho(X)-1)\cdot b(X)$ that was established above.
   Consider
   a smooth Kummer surface $X\subset\mathbb P^3$ with 16 disjoint lines
   (or with 16 disjoint smooth rational curves of some degree)
   as in \cite{BarNie94} or in \cite{Bau97}.
   The generic such surface has $\rho(X)=17$, we have $b(X)=-2$,
   and if $C$ is the union of the 16 disjoint curves, then
   $C^2=16\cdot(-2)$.
\end{example}

\begin{example}\rm
   A more elementary example is given by the blow up $X$ of ${\mathbb P}^2$
   at $n\leq 8$ general points, so $\rho(X)=n+1$. Since $-K_X$ is ample, it follows by adjunction
   for any reduced, irreducible curve $C$ that $C^2\geq-1$, so $b(X)=1$. But if $E$ is the union of
   the exceptional curves of the $n$ blown up points, then $E^2=-n=-(\rho(X)-1)\cdot b(X)$.
\end{example}



\bigskip
\small
   Tho\-mas Bau\-er,
   Fach\-be\-reich Ma\-the\-ma\-tik und In\-for\-ma\-tik,
   Philipps-Uni\-ver\-si\-t\"at Mar\-burg,
   Hans-Meer\-wein-Stra{\ss}e,
   D-35032~Mar\-burg, Germany.

\nopagebreak
   \textit{E-mail address:} \texttt{tbauer@mathematik.uni-marburg.de}

\bigskip
   Brian Harbourne,
   Department of Mathematics,
   University of Nebraska-Lincoln,
   Lincoln, NE 68588-0130, USA.

\nopagebreak
   \textit{E-mail address:} \texttt{bharbour@math.unl.edu}

\bigskip
   Andreas Leopold Knutsen,
   Department of Mathematics,
   University of Bergen,
   Johs. Brunsgt. 12,
   N-5008 Bergen, Norway.

\nopagebreak
   \textit{E-mail address:} \texttt{andreas.knutsen@math.uib.no}

\bigskip
   Alex K\"uronya,
   Budapest University of Technology and Economics,
   Mathematical Institute, Department of Algebra,
   Pf. 91, H-1521 Budapest, Hungary.

\nopagebreak
   \textit{E-mail address:} \texttt{alex.kuronya@math.bme.hu}

\medskip
   \textit{Current address:}
   Alex K\"uronya,
   Albert-Ludwigs-Universit\"at Freiburg,
   Mathematisches Institut,
   Eckerstra{\ss}e 1,
   D-79104 Freiburg,
   Germany.

\bigskip
   Stefan M\"uller-Stach,
   Institut f\"ur Mathematik (Fachbereich 08)
   Johannes Gutenberg-Universit\"at Mainz
   Staudingerweg 9
   55099 Mainz, Germany.

\nopagebreak
   \textit{E-mail address:} \texttt{stach@uni-mainz.de}

\bigskip
   Xavier Roulleau,
   D\'epartemento de Mathematica,
   Instituto Superior T\'ecnico,
   Avenida Rovisco Pais,
   1049-001 Lisboa,
   Portugal.

\nopagebreak
   \textit{E-mail address:} \texttt{roulleau@math.ist.utl.pt}

\bigskip
   Tomasz Szemberg,
   Instytut Matematyki UP,
   Podchor\c a\.zych 2,
   PL-30-084 Krak\'ow, Poland.

\nopagebreak
   \textit{E-mail address:} \texttt{tomasz.szemberg@uni-due.de}


\end{document}